\documentclass[11pt]{amsart}

\usepackage{amsmath,amssymb,amsfonts,amsthm,amscd,indentfirst}

\def\begeq{\begin{equation}}
\def\endeq{\end{equation}}

\title[A Rigidity Theorem for hypersurfaces]{A Rigidity Theorem for hypersurfaces in higher dimensional space forms}

\author{Pengfei Guan and Xi Sisi Shen}
\address{
        Department of Mathematics and Statistics\\
        McGill University\\
        Montreal, Quebec. H3A 2K6, Canada.}

\email{guan@math.mcgill.ca, xi.shen@mail.mcgill.ca}

\thanks{Research of the first author was supported in part by NSERC Discovery Grant.}

\date{}

\usepackage{graphicx}
\usepackage{amssymb,amsmath}
\usepackage{moreverb}
\usepackage{amsthm}
\newtheorem{lemma}{Lemma}

\newtheorem{thm}{Theorem}

\begin{document}
\begin{abstract}
We prove a rigidity theorem for hypersurfaces in space form $N^{n+1}(K)$, generalizing the classical Cohn-Vossen theorem.
\end{abstract}

\maketitle

The classical Cohn-Vossen theorem \cite{CV} states that two isometric compact convex surfaces in $\mathbb R^3$ are congruent. There is a vast literature devoted to the study of rigidity of hypersurfaces, a good source of reference is \cite{spivak}. In this short note, we prove a higher dimensional version of the Cohn-Vossen Theorem for hypersurfaces in space form $N^{n+1}(K)$, $n\ge 2$. The original convexity assumption in the Cohn-Vossen Theorem will be replaced by the assumption that hypersurfaces are star-shaped with normalized scalar curvature $R>K$. When $K=0$ and $n=2$, $N^{n+1}(K)=\mathbb R^{3}$ and the scalar curvature is the Gauss curvature. Positivity of Gauss curvature of $M$ implies the embedding is convex, it in turn is star-shaped with respect to any interior point. 

\medskip

The key ingredient in our proof of this higher dimensional generalization of Cohn-Voseen theorem is the integral formula (\ref{zero0}) we will establish. We will make crucial use of a conformal Killing field associated with the polar potential of the space form. 

Let $N^{n+1}(K)$ be a simply connected $(n+1)$-dimensional space form with constant curvature $K$. We will assume $K=-1,0$ or $+1$ and $n\ge 2$. Denote $g^N:=ds^2$ the Riemannian metric of $N^{n+1}(K)$. We will use the geodesic polar coordinates.
Let $\mathbb S^n$ be the unit sphere in Euclidean space $\mathbb R^{n+1}$ with standard induced metric $d\theta^2$, then
\begin{equation}
  g^N:= ds^2=d\rho^2+\phi^2(\rho)d\theta^2.
  \label{polarcoord}
\end{equation}
For the Euclidean space $\mathbb R^{n+1}$, $\phi(\rho)=\rho$, $\rho\in [0,\infty)$;
for the elliptic space $\mathbb S^{n+1}$, $\phi(\rho)=\sin(\rho)$, $\rho\in [0,\pi)$; and for the hyperbolic space $\mathbb H^{n+1}$, $\phi(\rho)=\sinh(\rho)$, $\rho\in [0,\infty)$.
We define the corresponding polar potential function $\Phi$ and a key vector field $V$ as
\begin{equation}
  \Phi(\rho)=\int^\rho_0 \phi(r)dr, \quad V=\phi(\rho)\frac{\partial}{\partial \rho}
  \label{Phi}
\end{equation}
Therefore, for the cases $K=-1,0, +1$, the corresponding polar potentials are $-\cos\rho$, $\frac{\rho^2}{2}$, and $\cosh\rho$, respectively. It is well known that $V$ on $N^{n+1}(K)$ is a conformal Killing field and it satisfies $D_iV_j=\phi'(\rho)g^N_{ij}$, where $D$ is the covariant derivative with respect to the metric $g^N$. Let $M\subset N^{n+1}(K)$ be a closed hypersurface with induced metric $g$ with outer normal $\nu$. We define support function as
\begin{equation}
u=<V,\nu>.\end{equation}
The hypersurface $M$ is star-shaped with respect to the origin if and only if $u>0$. The following identity (e.g., see \cite{GL}) will play an important role in our derivation.
\begin{equation}
    \nabla_{i}\nabla_j\Phi = \phi'(\rho) g_{ij}-h_{ij}u,
    \label{hessianPhi}
  \end{equation}
where $\nabla$ is the covariant derivative with respect to the induced metric $g$ and $h=(h_{ij})$ is the second fundamental form of the hypersurface. Denote $W=g^{-1}h$ the Weingarten tensor, define the $2^{nd}$ symmetric function of Weingarten tensor $W$ by
\begin{equation}\label{defsigma}
\sigma_{2}(W) =\sum_{i<j}(w_{ii}w_{jj}-w_{ij}w_{ji})=\sum_{i<j}\kappa_i \kappa_j,
\end{equation}
where $\kappa=(\kappa_1,\cdots,\kappa_n)$ are the eigenvalues of $W$ which are the principal curvatures of $M$. The relationships of the principal curvatures and the normalized scalar curvature $R$ of the induced metric from ambient space $N^{n+1}(K)$ are as follow.
\begin{equation}\label{scalarcurv}
\sigma_2(W)=\frac{n(n-1)}{2}(R-K).\end{equation}
Since $R$ is invariant under isometries, so is $\sigma_2(W)$. From definition (\ref{defsigma}),
\begin{eqnarray*}
\frac{\partial\sigma_{2}}{\partial w_{ij}}(W)=\sum_{l\neq i}w_{ll}, \quad \mbox{if $i=j$};\quad
\frac{\partial\sigma_{2}}{\partial w_{ij}}(W)=-w_{ji}, \quad \mbox{if $i\neq j$}.\end{eqnarray*}

Define
\begin{equation}\label{N+} N^{n+1}_{+}(K)=N^{n+1}(K), \quad \mbox{ $K=-1$ or $0$;} \quad N^{n+1}_{+}(K)=S^{n+1}_{+},  \mbox{$K=1$,}\end{equation}where $S^{n+1}_{+}$ is any open hemisphere. We will restrict ourselves to hypersurfaces in $N^{n+1}_{+}(K)$. Suppose $M$ and $\tilde M$ are two isometric connected compact hypersurfaces of $N^{n+1}(K)$ with the normalized $R> K$. We may identify any point $\tilde x\in \tilde M$ as a point $x\in M$ through the isometry. This identification will be used in the rest of this article. We may then choose an orthonormal frame on $M$ and by the isometry this will correspond to the same orthonormal frame on $\tilde{M}$. We may view $M$ as a base manifold, the local orthonormal frames of $M$ and $\tilde M$ can be identified as the same. Denote $W$ and $\tilde W$ the corresponding Weingarten tensors of $M$ and $\tilde M$, respectively. For any fixed local orthonormal frame on $M$, the polarization of $\sigma_2$
of two symmetric tensors $W=(w_{ij})$ and $\tilde W=(\tilde w_{ij})$ is defined as
\begin{equation}\label{defsigma2}
\sigma_{1,1}(W,\tilde W) = \frac{1}{2}\sum_{i,j} \frac{\partial \sigma_{2}(W)}{\partial w_{ij} }\tilde{w}_{ij}
\end{equation}

The following lemma is a special case of the Garding inequality \cite{garding}.
\begin{lemma}\label{garding}
Given two isometric compact hypersurfaces $M$ and $\tilde{M}$ in $N^{n+1}_{+}(K)$ with the normalized scalar curvature $R>K$, then
\[\sigma_{2}(W)(x)-\sigma_{1,1}(W,\tilde W)(x)\le 0, \quad \forall x\in M.\] If the equality holds at some point $x$, then $W(x)=\tilde W(x)$.
\end{lemma}
\begin{proof} From identity (\ref{scalarcurv}), $\sigma_2(W)=\sigma_2(\tilde W)>0$. By the compactness,
there are points $p\in M$ and $\tilde p\in \tilde M$ such that $W(p)>0$ and $\tilde W(\tilde p)>0$. This implies $W(x), \tilde W(x)\in \Gamma_2, \forall x\in M$, where $\Gamma_2=\{\sigma_1(W)>0, \sigma_2(W)>0\}$ is the Garding cone. The lemma follows from the Garding inequality \cite{garding}.\end{proof}

Suppose $g\in C^3$ and suppose $e_{1},\ldots, e_{n}$ is a local orthonormal frame on $M$ and $W=(w_{ij})$ is a Codazzi tensor on $M$, then for each $i$,
\begin{eqnarray}\label{Codazzi}
\sum_{j=1}^{n}(\frac{\partial\sigma_{2}}{\partial w_{ij}})_{j}(W)&=& \sigma_2^{ii}(W)_i+\sum_{j\neq i}\sigma_2^{ij}(W)_j\\
&=& (\sum_{l=1}^n w_{ll,i} -w_{ii,i})-\sum_{j\neq i}w_{ji,j}\nonumber\\
&=& \sum_{l=1}^n w_{ll,i} -w_{ii,i}-\sum_{j\neq i}w_{jj,i}=0.\nonumber
\end{eqnarray}

\medskip

We will establish the integral formulae needed.
\begin{lemma}\label{integralformula}
Suppose $M$ and $\tilde M$ are two isometric $C^2$ star-shaped hypersurfaces with respect to origin in the polar coordinates of the ambient space $N^{n+1}(K)$ in (\ref{polarcoord}). Identify any local frame on $M$ with a local frame on $\tilde M$ via isometry. Denote $\Phi$ and $\tilde \Phi$, and $u$ and $\tilde u$ to be the polar potential functions and support functions of $M$ and $\tilde M$, respectively. Then
\begin{eqnarray}\label{zero0}
\quad \quad \left\{
\begin{matrix}
K\int_{M} \sum_{i,j}\sigma_{2}^{ij}(W)\tilde\Phi_{j}\Phi_{i} &=& \int_{M} [\tilde \phi'\phi'\sigma_{1}(W)-2\tilde \phi'u\sigma_{k}(W)],\\
K\int_{M} \sum_{i,j}\sigma_{2}^{ij}(\tilde W)\tilde\Phi_{j}\Phi_{i} &=& \int_{M} [\tilde \phi'\phi'\sigma_{1}(\tilde W)-2\tilde \phi'u\sigma_{1,1}(W,\tilde W)],\\
K\int_{M} \sum_{i,j}\sigma_{2}^{ij}(W)\Phi_{j}\tilde\Phi_{i} &=& \int_{M} [\phi'\tilde \phi'\sigma_{1}(W)-2\phi'\tilde u\sigma_{1,1}(\tilde W,W)],\\
K\int_{M} \sum_{i,j}\sigma_{2}^{ij}(\tilde W)\Phi_{j}\tilde \Phi_{i} &=& \int_{M} [\phi'\tilde \phi'\sigma_{1}(\tilde W)-2\phi'\tilde u\sigma_{2}(\tilde W)].\end{matrix}\right.
\end{eqnarray}\end{lemma}
\begin{proof}
We first assume $M$ and $\tilde M$ are $C^3$ star-shaped hypersurfaces.
For any local frame $\{e_1,\cdots, e_n\}$ on $M$, it is also a local frame on $\tilde M$. By the assumption of isometry, $g_{ij}=\tilde g_{ij}$.  We note that, with $\phi$ and $\Phi$ defined in (\ref{Phi}),
\begin{equation}\label{phiPhi}
\nabla_i \Phi=-K\nabla_i \phi^{'}, \quad \forall i=1,\cdots, n.\end{equation}
Denote $\tilde \phi=\phi(\tilde \rho), \tilde\phi'=\phi'(\tilde\rho), \tilde\Phi=\Phi(\tilde\rho)$. It follows from (\ref{hessianPhi}),
\begin{eqnarray}
   \tilde\phi' \nabla_{i}\nabla_j\Phi =\tilde\phi'(\phi' g_{ij}-h_{ij}u),\quad
      \phi' \nabla_{i}\nabla_j\tilde\Phi =\phi'(\tilde \phi' \tilde g_{ij}-\tilde h_{ij}\tilde u)\nonumber.
\label{hessianPhi0}
  \end{eqnarray}
Contracting $\sigma_{2}^{ij}(W)$ and $\sigma_2^{ij}(\tilde W)$ with equations in above, (\ref{zero0}) follows from identities (\ref{Codazzi}), (\ref{phiPhi}) and integration by parts.

(\ref{zero0}) for $C^2$ hypersurfaces can be verified by apprixmation. We may approximate them by $C^3$ hypersurfaces $M^{\epsilon}$ and $\tilde M^{\epsilon}$. Note that $M^{\epsilon}$ and $\tilde M^{\epsilon}$ may not be isometric. For any local frame on $M$,  the following still holds,
\begin{eqnarray}
   \tilde\phi^{\epsilon'} \nabla_{i}\nabla_j\Phi^{\epsilon} =\tilde\phi^{\epsilon'}(\phi^{\epsilon'} g^{\epsilon}_{ij}-h^{\epsilon}_{ij}u^{\epsilon}),\quad
      \phi^{\epsilon'} \nabla_{i}\nabla_j\tilde\Phi^{\epsilon} =\phi^{\epsilon'}(\tilde \phi^{\epsilon'} \tilde g^{\epsilon}_{ij}-\tilde h^{\epsilon}_{ij}\tilde u^{\epsilon})\nonumber.
\label{hessianPhi10}
  \end{eqnarray}
Using $(g^{\epsilon})^{-1}$ to contract with $\sigma_2^{ij}(W^{\epsilon})$ and using $(\tilde g^{\epsilon})^{-1}$ to contract with $\sigma_2^{ij}(\tilde W^{\epsilon})$, we may perform integration by parts as before at $\epsilon$-level. By (\ref{Codazzi}) and (\ref{phiPhi}),
\begin{eqnarray*}
K\int_{M} \sum_{i,j}\sigma_{2}^{ij}(W^{\epsilon})\tilde\Phi^{\epsilon}_{j}\Phi^{\epsilon}_{i}dV_{g^{\epsilon}} &=& \int_{M} [\tilde \phi^{\epsilon'}\phi^{\epsilon'}\sigma_{1}(W^{\epsilon})-2\tilde \phi^{\epsilon'}u^{\epsilon}\sigma_{k}(W^{\epsilon})]dV_{g^{\epsilon}},\\
K\int_{M} \sum_{i,j}\sigma_{2}^{ij}(\tilde W^{\epsilon})\Phi^{\epsilon}_{j}\tilde \Phi^{\epsilon}_{i}dV_{\tilde g^{\epsilon}} &=& \int_{M} [\phi^{\epsilon'}\tilde \phi^{\epsilon'}\sigma_{1}(\tilde W^{\epsilon})-2\phi^{\epsilon'}\tilde u^{\epsilon}\sigma_{2}(\tilde W^{\epsilon})]dV_{\tilde g^{\epsilon}},\\
K\int_{M} \sum_{i,j}\sigma_{2}^{ij}(\tilde W^{\epsilon})\tilde\Phi^{\epsilon}_{j}\Phi^{\epsilon}_{i}dV_{\tilde g^{\epsilon}} &=& \int_{M} [\tilde \phi^{\epsilon'}\phi^{\epsilon'}\sigma_{1}(\tilde W^{\epsilon})-2\tilde \phi^{\epsilon'}u^{\epsilon}\sigma_{1,1}(W^{\epsilon},\tilde W^{\epsilon})]dV_{\tilde g^{\epsilon}}\\
&&+\mathcal{R}^{\epsilon}_1,\\
K\int_{M} \sum_{i,j}\sigma_{2}^{ij}(W^{\epsilon})\Phi^{\epsilon}_{j}\tilde\Phi^{\epsilon}_{i}dV_{g^{\epsilon}} &=& \int_{M} [\phi^{\epsilon'}\tilde \phi^{\epsilon'}\sigma_{1}(W^{\epsilon})-2\phi^{\epsilon'}\tilde u^{\epsilon}\sigma_{1,1}(\tilde W^{\epsilon},W^{\epsilon})]dV_{g^{\epsilon}}\\
&&+\mathcal{R}^{\epsilon}_2,\end{eqnarray*}
where the error terms $\mathcal{R}^{\epsilon}_i, i=1,2$ involves only the differences of derivatives of $g^{\epsilon}$ and $\tilde g^{\epsilon}$ up to second order. Therefore, $\mathcal{R}^{\epsilon}_i \to 0$  as $\epsilon\to 0$ for $i=1,2$.  (\ref{zero0}) follows for $C^2$ isometric star-shaped hypersurfaces by letting $\epsilon\to 0$.
\end{proof}

\medskip

With integra formulae (\ref{zero0}), we follow the similar argument of Herglotz in \cite{herglotz} (see also \cite{hopf}) and using the Garding's inequality as in \cite{chern} to prove that two isometric star-shaped compact hypersurfaces in $N^{n+1}(K)$ share the same second fundamental form.

We now proceed to prove the main result below.
The term \emph{congruency} will be used to describe two hypersurfaces in $N^{n+1}(K)$ that differ by an isometry of the ambient space.

\begin{thm} Two $C^2$ isometric compact star-shaped hypersurfaces in $N^{n+1}(K)$ with the normalized scalar curvature $R$ strictly larger than $K$ are congruent if $K=-1, 0$. Two $C^2$ isometric compact star-shaped hypersurfaces in $\mathbb S^{n+1}$ with normalized scalar curvature strictly larger than $+1$ are congruent if the hypersurface are contained in some (may be different) hemispheres.
\end{thm}

\begin{proof} We may assume that $M$ and $\tilde{M}$ are two star-shaped hypersurfaces with respect to a fixed point $p\in N^{n+1}(K)$. We may use polar coordinates in (\ref{polarcoord}) and assume $p=0$.
In this setting, $M$ and $\tilde{M}$ are in the same $N^{n+1}_{+}(K)$.

Subtracting the first equation in (\ref{zero0}) from the third and the second from the fourth,
\begin{eqnarray*}
\int_{M} \tilde\phi'u\sigma_{2}(W)&=&\int_{M} \phi'\tilde u\sigma_{1,1}(W,\tilde{W}),\\
\int_{M} \phi'\tilde{u}\sigma_{2}(\tilde{W})&=&\int_{M} \tilde\phi'u\sigma_{1,1}(\tilde{W},W)
\end{eqnarray*}
As $\sigma_2(W)=\sigma_2(\tilde W)$ and $\sigma_{1,1}(W,\tilde W)=\sigma_{1,1}(\tilde W, W)$, it follows that
\begin{equation}\label{zero}
\int_{M} (\tilde\phi'u+\phi'\tilde{u})(\sigma_{2}(W)-\sigma_{1,1}(W,\tilde W) )= 0.
\end{equation}
Since the support functions $u$ and $\tilde{u}$ are strictly positive and the hypersurfaces are in the same $N^{n+1}_{+}(K)$, $\phi'(x)\tilde \phi'(x)>0, \forall x\in M$.
That is, $\tilde\phi'u+\phi'\tilde{u}$ is nowhere vanishing on $M$.
We conclude from (\ref{zero}) and Lemma \ref{garding} that on $M$,
\[\sigma_{2}(W)-\sigma_{1,1}(W,\tilde W)\equiv 0.\] Again by Lemma \ref{garding}, $W=\tilde{W}$ on $M$. Thus, the first and second fundamental forms of $M$ and $\tilde M$ are the same. This implies that $M$ and $\tilde M$ are congruent.
\end{proof}

\end{document}